\documentclass{ifacconf}

\usepackage{graphicx}      % include this line if your document contains figures
\graphicspath{{figures/}} % path for figures etc
\usepackage{natbib}        % required for bibliography

\usepackage{arydshln}
\usepackage{amsmath}
\usepackage{mathdots}
\usepackage{mathrsfs}  % for \mathscr
\usepackage{amssymb} % symbols as \prec \succ etc
\usepackage{arydshln}% dashed horizontal and vertical lines
\usepackage{graphicx} % graphic support
\usepackage{xcolor}
\usepackage{tikz}
\usepackage[absolute]{textpos} % For "Watermark" on first page
\usepackage{mathtools}
\usepackage[hidelinks]{hyperref}
\newcommand{\R}{\mathbb{R}}         % real numbers
\newcommand{\N}{\mathbb{N}}         % natural numbers
\renewcommand{\S}{\mathbb{S}}          % symmetric matrices
\newcommand{\Dc}{\mathcal{D}}

\newcommand{\Pb}{\mathbf{P}}

\newcommand{\cge}{\succcurlyeq}
\newcommand{\cle}{\preccurlyeq}

\newcommand{\cg}{\succ}
\newcommand{\kron}{\otimes}         % kronecker product
\newcommand{\Delf}{\mathbf{\Delta}}

\newcommand{\diag}{\mathrm{diag}}

\newcommand{\col}{\mathrm{col}}

\newcommand{\Del}{\Delta}
\newcommand{\del}{\delta}
\newcommand{\ga}{\gamma}

\newcommand{\eps}{\varepsilon}
\newcommand{\la}{\lambda}

\newcommand{\mat}[2]{\left(\begin{array}{@{}#1@{}}#2\end{array}\right)} %ersetze #1 durch @{}#1@{} fuer weniger platz rechts und links
\newcommand{\smat}[1]{\left(\begin{smallmatrix}#1\end{smallmatrix}\right)}

\newcommand{\teq}[1]{\quad\text{#1}\quad} % for text in math enironments
\newenvironment{red_test}{\color{red}}{} % mark red

\renewcommand{\t}{\tilde}
\newcommand{\h}{\hat}
\newcommand{\ch}{\check}

\newcommand{\mstrut}[1]{\rule{0pt}{#1}}
\let\oldhline=\hline % remember original command
\renewcommand{\hline}{\oldhline\mstrut{2.5ex}}
\let\oldhdashline=\hdashline % remember original command
\renewcommand{\hdashline}{\oldhdashline\mstrut{2.5ex}}

\definecolor{uniSMittelblau}{RGB}{0,81,158}
\definecolor{uniSHellblau}{RGB}{0,190,255}
\usepackage{titlesec}
\titleformat{\section}
{\color{uniSMittelblau}\centering}
{\color{uniSMittelblau}\thesection.}{1ex}{}

\titleformat{\subsection}
{\color{uniSHellblau}\itshape}
{\color{uniSHellblau}\thesubsection}{1ex}{}

\usepackage{amsthm,xcolor}

\newtheoremstyle{mystyle}%
{3pt}% Space above
{3pt}% Space below
{}% Body font
{}% Indent amount
{\itshape\color{uniSMittelblau}}% Theorem head font
{.}% Punctuation after theorem head
{.5em}% Space after theorem head
{}% Theorem head spec (can be left empty, meaning ‘normal’)

\theoremstyle{mystyle}
\newtheorem{theo}{Theorem}%[section]
\newtheorem{lemm}[theo]{Lemma}
\newtheorem{coro}[theo]{Corollary}
\newtheorem{defi}[theo]{Definition}
\newtheorem{assu}[theo]{Assumption}
\newtheorem{rema}[theo]{Remark}

\renewenvironment{proof}[1][Proof]{%\vspace{1.5ex}
	\bf #1. \rm}%{}%
{\hfill \footnotesize{$\blacksquare$}\vspace{2ex}}

\begin{document}
	\begin{frontmatter}
		
		\title{\textcolor{uniSMittelblau}{Input-Output-Data-Enhanced Robust Analysis via Lifting}\thanksref{footnoteinfo}}
		
		\thanks[footnoteinfo]{Funded by Deutsche Forschungsgemeinschaft (DFG, German Research Foundation) under Germany’s Excellence Strategy - EXC 2075 - 390740016. We acknowledge the support by the Stuttgart Center for Simulation Science (SimTech).}
		
		\author[First]{Tobias Holicki and Carsten W. Scherer}
		
		\address[First]{Department of Mathematics, University of Stuttgart, Germany,
			e-mail: \{tobias.holicki, carsten.scherer\}@imng.uni-stuttgart.de
		}

		\begin{abstract}               % Abstract of not more than 250 words.					
			Starting from a linear fractional representation of a linear system affected by constant parametric uncertainties, we demonstrate how to enhance standard robust analysis tests by taking available (noisy) input-output data of the uncertain system into account. Our approach relies on lifting the system and the construction of data-dependent multipliers.
			It leads to a test in terms of linear matrix inequalities which guarantees stability and performance for all systems compatible with the observed data if it is in the affirmative.
			In contrast to many other data-based approaches, prior physical or structural knowledge about the system can be incorporated at the outset by exploiting the power of linear fractional representations.
		\end{abstract}

		
	\end{frontmatter}
	%===============================================================================
	
	%\begin{textblock}{9}(1.15, 16)
	%	This work has been submitted to IFAC for possible publication.
	%\end{textblock}
	
	\begin{textblock}{13}(1.55, 15.75)
		\fbox{
			\small \textcopyright~ 2022 the authors. This work has been accepted to IFAC for publication under a Creative Commons Licence CC-BY-NC-ND}
	\end{textblock}
	
	%	\begin{textblock}{12}(1.75, 15.75)
	%		\fbox{
	%			\begin{minipage}{\textwidth}
	%				\small \textcopyright~ 2022 IFAC. This work has been published in IFAC-PapersOnline under a Creative Commons Licence CC-BY-NC-ND. \\ DOI: 10.1016/j.ifacol.2020.12.981
	%			\end{minipage}
	%		}
	%	\end{textblock}

	%-----------------------------------------------------------------------------------------
	\section{INTRODUCTION}
	%-----------------------------------------------------------------------------------------
	
	In recent years direct data-driven analysis and control has gained a lot of attention even for linear time-invariant (LTI) systems \citep{HouWan13, MarDoe21}.
	The general theme is to leverage gathered data of an unknown system for its analysis or controller design. 
	Some of the key challenges are
	\begin{itemize}
		\item to provide strong theoretical \emph{guarantees} despite the fact that merely a finite amount of data points is available that are typically affected by noise.
		\item to systematically exploit \emph{prior physical or structural knowledge} which is available in almost any practical application and  typically required for data efficiency.
	\end{itemize}
	For example, the approaches in \citep{DeTes20} and \citep{WaaCam22} can provide such guarantees under suitable assumptions and even if merely a single trajectory of the system is available. These approaches rely on the fundamental lemma of Willems and a matrix version of the classical S-procedure, respectively. However, neither of them does take prior knowledge into account.

	The framework of linear fractional representations (LFRs) as discussed, e.g., by \cite{ZhoDoy98} and \cite{DoyPac91} is widely acknowledged as a powerful tool to systematically model dynamical systems affected by various types of uncertainties in a such a way that a priori known components are nicely separated from the unknown ones.
	Surprisingly, only few data-driven approaches exploit the LFR framework in order to incorporate prior knowledge, such as \citep{BerSch15, MarHen17, FieSch21, HolSch21}.
	Alternative approaches to include prior information can be found, e.g., in \citep{KobBag13, RohTri18}. However, in our view, LFRs offer the greatest modelling flexibility while permitting the use of well-developed and dedicated robust analysis tools.

	This paper is closely related to the one by \cite{BerSch20}, which systematically incorporates prior knowledge about the describing matrices of an otherwise unknown LTI system in order to design controllers merely based on one trajectory of the system. It relies on LFRs for incorporating prior information and employs multiplier based robust control for robustness analysis and controller design. The key ingredient is the construction of dedicated multipliers that take the gathered data into account.
	In particular, due to the applied robust control techniques, it yields guarantees for all systems that are compatible with the observed data.

	In contrast to \citep{BerSch20} which mostly deals with data in terms of an input-state trajectory, we consider the more challenging case that merely an input-output trajectory is available.
	To the best of our knowledge, all related approaches involving such input-output data revolve around systems described by higher order difference equations \citep{BerSch20, DeTes20, WaaEis22}. 
	Instead, we consider a system with a genuine LFR in which the unknown parts are viewed as parametric uncertainties. It is our expectation that this is instrumental for moving towards systems involving more general types of unknown parts. 
	
	\vspace{1ex}
	%-----------------------------------------------------------------------------------------
	\noindent\textit{Outline.} %
	%-----------------------------------------------------------------------------------------
	%
	The paper is organized as follows. After a short paragraph on notation, we  recall a classical robust analysis result and show how this test can be enhanced once noise-free input-output data are available in Subsections \ref{IOD::sec::basic_ana} and \ref{IOD::sec::main_noisefree}, respectively. In the rest of Section~\ref{IOD::sec::simplest_setting}, we elaborate on the assumptions involved in our main result and on how to relax them.
	Section~\ref{IOD::sec::noisy_data} is structured similarly and deals with the situation that the input-output trajectory is affected by noise.
	We conclude with a numerical example and some further remarks in Sections \ref{IOD::sec::example} and \ref{IOD::sec::conclusions}, respectively.

	\vspace{1ex}
	%-----------------------------------------------------------------------------------------
	\noindent\textit{Notation.} %
	%-----------------------------------------------------------------------------------------
	%
	Let $\ell_2^n\!:=\! \{(x(k))_{k=0}^\infty\!:  \sum_0^\infty \!x(k)^{\!\top}\! x(k)\!<\! \infty \}$ be the space of square summable sequences with elements in $\R^n$.
	Moreover, $\S^n$ denotes the set of symmetric $n\times n$ matrices.
	Finally, we use the abbreviation
	\begin{equation*}
		\arraycolsep=1pt
		\diag(X_1, \dots, X_N) := \smat{X_1 & & 0 \\[-1ex] & \ddots & \\ 0 & & X_N}
	\end{equation*}
	for matrices $X_1, \dots, X_N$, utilize the Kronecker product $\kron$ as defined in \citep{HorJoh91}
	and indicate objects that can be inferred by symmetry or are not relevant with the symbol ``$\bullet$''.

	%-----------------------------------------------------------------------------------------
	\section{NOISE-FREE DATA}\label{IOD::sec::simplest_setting}
	%-----------------------------------------------------------------------------------------

	For real matrices of appropriate dimensions and some initial condition $x(0) \in \R^n$, let us consider the discrete-time feedback interconnection
	\begin{equation}
		\label{eq::sys_simple}
		\mat{c}{x(k+1) \\   z(k) \\  y(k)}
		\!=\! \mat{ccc}{A & B_w & B_r  \\
			C_z & D_{zw} & D_{zr}  \\
			C_y & D_{yw} & D_{yr}}\!
		\mat{c}{x(k) \\ w(k) \\  r(k)},~~
		w(k) = \Del_\ast z(k)
	\end{equation}
	with interconnection signals $w$, $z$, an input $r$ to excite the system and a measurable output $y$. Moreover, this interconnection also involves a constant parametric uncertainty $\Del_\ast$ that is merely known to be contained in some given set $\Delf \subset \R^{n_w \times n_z}$ of potentially highly structured matrices.
	
	%-----------------------------------------------------------------------------------------
	\subsection{Linear Fractional Representations}
	%-----------------------------------------------------------------------------------------

	The description \eqref{eq::sys_simple} constitutes a linear fractional representation (LFR) which is a well-established modeling tool in the robust control community and widely used in practice \citep{ZhoDoy98}.
	These LFRs are usually employed in tandem with dedicated robust analysis tests \citep{DoyPac91} that take the known part of \eqref{eq::sys_simple} and the information on the uncertain part encoded in the set $\Delf$ into account.
	Frequently, $\Delf$ equals
	\begin{equation*}
		\left\{\diag(\del_1 I, \dots, \del_{m_r}I, \Del_1, \dots, \Del_{m_f})\,:\, |\del_j| \leq 1,~ \|\Del_j\|\leq 1 \right\}\!,
	\end{equation*}
	i.e., it consists of structured matrices with (repeated) diagonal and unstructured blocks on the diagonal that are all bounded in norm by one. Classically, one also considers dynamic uncertainties in the interconnection \eqref{eq::sys_simple}, but we restrict our attention to constant parametric ones.
	
	Let us note that the description \eqref{eq::sys_simple} trivially includes settings in which no prior knowledge at all is taken into account, such as in the work by \cite{DeTes20}. If we suppose that the full state can be measured, this corresponds to considering a system given by
	\begin{equation*}
		\mat{c}{x(k+1) \\ y(k)} = \mat{cc}{A_\ast & B_\ast \\ I & 0}\mat{c}{x(k) \\ r(k)}
	\end{equation*}
	with unknown matrices $A_\ast \in \R^{n\times n}$ and $B_\ast \in \R^{n \times n_r}$.
	Indeed, this can be subsumed to \eqref{eq::sys_simple} by choosing
	\begin{equation*}
		\mat{c|c:c}{A & B_w & B_r  \\ \hline
			C_z & D_{zw} & D_{zr}  \\ \hdashline
			C_y & D_{yw} & D_{yr}}
		= \mat{c|c:c}{0 & I & 0 \\ \hline \smat{I \\ 0} & 0 & \smat{0 \\ I} \\ \hdashline I & 0 & 0},\quad
		\Del_\ast = \mat{cc}{A_\ast & B_\ast}
	\end{equation*}
	and $\Delf = \R^{n \times (n+n_r)}$. 
	The numerical example in Section~\ref{IOD::sec::example} illustrates the power of LFRs to capture more sophisticated structural properties.

	%-----------------------------------------------------------------------------------------
	\subsection{A Basic Robust Analysis Result}\label{IOD::sec::basic_ana}
	%-----------------------------------------------------------------------------------------

	In order to analyze stability and performance properties of the interconnection \eqref{eq::sys_simple}, we rely on the full-block S-procedure as discussed in \citep{Sch00}.
	To this end, we require so-called multiplier sets and, for later purposes, dual variants thereof.	
	
	\begin{defi}
		$\Pb \subset \S^{n_z + n_w}$ is called a (dual) \emph{multiplier set} for the set of uncertainties  $\Delf\subset \R^{n_w \times n_z}$ if 
		\begin{equation*}
			\mat{c}{-\Del^\top \\ I}^\top P \mat{c}{-\Del^\top \\ I} \cle 0
			\teq{ for all }(\Del, P) \in \Delf \times \Pb
		\end{equation*}
		holds and if $\Pb$ admits an LMI representation, i.e., there exist affine matrix-valued maps $E, F$ such that $\Pb = \{E(\nu)\,:\, \nu \in \R^\bullet \text{ and }F(\nu) \cg 0\}$.
	\end{defi}

	A detailed discussion with concrete choices for multiplier sets can be found in \citep{Sch00,Sch05}. Let us now recall the following classical robust stability result.
	
	\begin{lemm}
		\label{IOD::lem::basic}
		Let $\Pb$ be a multiplier set for $\Delf$. If there exist %$X\in \S^n$ and $P\in \S^{n_z + n_w}$ satisfying
		\begin{subequations}
			\begin{equation}
				X \cg 0 \teq{ and }P \in \Pb
			\end{equation}
			satisfying
			\begin{equation}
				(\bullet)^\top \mat{cc|c}{X & 0 & \\ 0 & -X & \\ \hline & & P}\mat{cc}{I & 0 \\ -A^\top & -C_z^\top \\ \hline 0 & I \\ -B_w^\top & -D_{zw}^\top} \cg 0
			\end{equation}
		\end{subequations}
		then the interconnection \eqref{eq::sys_simple} is stable, i.e., $I - D_{zw}\Del_\ast$ is nonsingular and there exist $M, \ga > 0$ such that $\|x(t)\| \leq Me^{-\ga t}\|x(0)\|$ for all $t \geq 0$, all $x(0)\in \R^n$ and for $r  = 0$.
	\end{lemm}
	
	\vspace{1ex}

	Note that Lemma \ref{IOD::lem::basic} does not properly take into account that the uncertainty in the interconnection \eqref{eq::sys_simple} is constant and, hence, can result in a conservative stability test.
	It can be improved, e.g., by employing integral quadratic constraints with dynamic multipliers \citep{MegRan97}, but we do not pursue this possibility here.
	Instead, we enhance this test by incorporating experimentally gathered data of the interconnection \eqref{eq::sys_simple}.
	
	%-----------------------------------------------------------------------------------------
	\subsection{Incorporating Data}\label{IOD::sec::main_noisefree}
	%-----------------------------------------------------------------------------------------
	
	More precisely, our goal is to improve the stability test in Lemma \ref{IOD::lem::basic} based on the availability of an input-output trajectory of the interconnection \eqref{eq::sys_simple} initialized in a known point $x_\ast$ and on a finite time-horizon of length $h$:
	\begin{equation}\label{data}
		\Dc := \big(x_\ast, (r_\ast(k))_{k=0}^{h-1}, (y_\ast(k))_{k=0}^{h-1}\big)\in \R^n \times \R^{hn_r} \times \R^{hn_y}.
	\end{equation}
	The assumption about $x_\ast$ will be relaxed in Subsection \ref{IOD::sec::unknown_initial_condition}.
	Following \cite{BerSch20}, we will use this data to learn more information about the unknown $\Del_\ast$ beyond the available prior knowledge encoded in $\Delf$, and exploit it for analyzing stability.
	Formally, we introduce the set
	\begin{equation*}
		\Delf_\Dc := \left\{\Del: \begin{array}{l}\text{The measured output of \eqref{eq::sys_simple} with} \\ \text{$x(0) = x_\ast$ and $\Del_\ast$ replaced by $\Del$} \\ \text{in response to the input $r_\ast$ is $y_\ast$}
		\end{array} \right\}
	\end{equation*}
	of all constant parametric uncertainties in $\R^{n_w \times n_z}$ that are compatible with the observed data $\Dc$.
	Clearly, we have
	\begin{equation*}
		\Del_\ast \in \Delf \cap \Delf_\Dc \subset \Delf.
	\end{equation*}
	Our main result will assure stability of the interconnection \eqref{eq::sys_simple} for all parametric uncertainties in $\Delf \cap \Delf_\Dc$. This is expected to be highly beneficial since the latter set is often much smaller than $\Delf$.

	%-----------------------------------------------------------------------------------------
	\subsection{Main Result}
	%-----------------------------------------------------------------------------------------

	The approaches in \citep{BerSch20} and \citep{DeTes20} conceptually rely on stacking the vectors in the available input-state data $\big((r_\ast(k))_{k=0}^{h-1}, (x_\ast(k))_{k=0}^{h}\big)$ horizontally,	which leads to a convenient linear equation in the unknown matrix $\Del_\ast$ in the noise-free case. Concretely,	one introduces the matrices
	\begin{equation*}
		X_+ := \mat{ccc}{x_\ast(1) & \dots & x_\ast(h)}, \quad
		X := \mat{ccc}{x_\ast(0) & \dots & x_\ast(h-1)}
	\end{equation*}
	as well as
	\begin{equation*}
		R := \mat{ccc}{r_\ast(0) & \dots & r_\ast(h-1)}
	\end{equation*}
	and observes from \eqref{eq::sys_simple} that
	\begin{equation}
		\label{IOD::data_based_equation_state_feedback}
		\begin{aligned}
			X_+ &= \big[A + B_w\Del_\ast C_z\big] X+ \big[B_r + B_w\Del_\ast D_{zr}\big]R \\
			&= \big[ AX + B_rR\big] + B_w \Del_\ast \big[C_zX + D_{zr}R \big]
		\end{aligned}
	\end{equation}
	holds if $D_{zw}$ is assumed to vanish; in \citep{DeTes20}, this equation even simplifies to
	\begin{equation*}
		X_+ = \Del_\ast \smat{X \\ R}.
	\end{equation*}
	In the case of genuine input-output data, it seems to be no longer possible to obtain such an %convenient
	equation in $\Del_\ast$ due to the dynamics involved in the interconnection \eqref{eq::sys_simple}.

	In our work, we employ lifting as exposed, e.g., by \cite{CheFra95}, which corresponds to stacking signals vertically as follows.
	
	\begin{lemm}
		The discrete-time lifting operator
		\begin{equation*}
			\h \cdot\,:\,\ell_2^{n_s} \to \ell_2^{hn_s},~~ \h s(k) = \mat{c}{s(hk) \\ \vdots \\ s(hk + h-1)}
		\end{equation*}
		is an isometric isomorphism.
	\end{lemm}

	Due to the variations of constants formula, we can then express the interconnection \eqref{eq::sys_simple} equivalently as
	\begin{equation}
		\label{IOD::eq::sys_simple_lifted}
		\mat{c}{\t x(k\!+\!1) \\   \h z(k) \\  \h y(k)}
		\hspace{-0.6ex}=\hspace{-0.6ex} \mat{ccc}{A^h & \h B_{hw} & \h B_{hr}  \\
			\h C_{hz} & \h D_{hzw}  &  \h D_{hzr}  \\
			\h C_{hy} & \h D_{hyw}  & \h D_{hyr}}\hspace{-1ex}
		\mat{c}{\t x(k) \\ \h w(k) \\  \h r(k)}\hspace{-0.7ex},~
		\h w(k) \!=\! \h \Del_h \h z(k)	
	\end{equation}
	with a new state defined through $\t x(k) := x(hk)$ and, for $i \in \{ w, r\}$ and $o\in \{z, y\}$, the lifted matrices
	\begin{equation*}
		\mat{c|c}{\bullet & \h B_{hi} \\ \hline
			\!\h C_{ho} & \h D_{hoi}}
		\hspace{-0.7ex}:=\hspace{-0.6ex} \mat{c|cccc}{\bullet & A^{h-1}B_i & A^{h-2}B_i & \dots & B_i \\ \hline
			C_o & D_{oi} & 0 & \dots & 0 \\
			C_oA & C_oB_i & D_{oi} & \ddots & \vdots \\
			\vdots & \vdots & \ddots & \ddots & 0 \\
			C_oA^{h-1} & C_oA^{h-2}B_i & \dots  & C_oB_i  & D_{oi}}\hspace{-0.8ex};
	\end{equation*}
	here and in the sequel, we further employ the abbreviations
	\begin{equation*}
		\h \Del_h := I_h \kron \Del_\ast
		\teq{ and }
		\ch \Del_h := \h \Del_h (I - \h D_{hzw} \h \Del_h)^{-1}.
	\end{equation*}
	After closing the loop in \eqref{IOD::eq::sys_simple_lifted} for 
	the trajectory generating the data $\Dc=(x_*,r_*,y_*)\in\R^n\times\R^{hn_r}\times\R^{hn_y}$, we obtain for $k=0$
	%and with $k=0$ , we note that $\hat r_*(0)=(r_*(i))_{i=0}^{h-1}=r_*$ and
	%Given the data $\Dc$, note that $\hat a$ In particular, for $k = 0$, if recalling the given data $\Dc$ and we obtain
	the following structured system of equations in the unknown $\Del_\ast$:
	\begin{equation}
		\label{IOD::eq::data_eq1}
		\arraycolsep=4pt
		\begin{aligned}
			0 &= y_\ast  \!-\!\big[\h C_{hy}\! +\! \h D_{hyw}\ch \Del_h\h C_{hz} \big]x_\ast \!-\!\big[\h D_{hyr} \!+\! \h D_{hyw}\ch \Del_h\h D_{hzr} \big] r_\ast \\
			&= \big[y_\ast - \h C_{hy}x_\ast - \h D_{hyr}r_\ast \big] - \h D_{hyw} \ch \Del_h \big[\h C_{hz}x_\ast + \h D_{hzr}r_\ast \big] \\
			&= \mat{cc}{\h D_{hyw}\ch \Del_h & I}\mat{ccc}{0 & - \h C_{hz} & -\h D_{hzr} \\ I& - \h C_{hy} & -\h D_{hyr}}\mat{c}{y_\ast \\ x_\ast \\ r_\ast}.
		\end{aligned}
	\end{equation}

	Building upon this data-based equation we introduce the one-parameter family of symmetric matrices
	\begin{equation}
		\Pb_\Dc := \left\{(\bullet)^\top q\mat{c}{y_\ast \\ x_\ast \\ r_\ast}^{\hspace{-1.2ex}\top}\hspace{-1ex} \mat{cc}{0 & I \\ -\h C_{hz}^\top & -\h C_{hy}^\top \\ -\h D_{hzr}^\top & -\h D_{hyr}^\top} ~\middle|~   q\in \R \right\}
		\label{IOD::eq::multiplier_noise_free}
	\end{equation}
	and conclude that any $P_\Dc \in \Pb_\Dc$ satisfies
	\begin{multline*}
		(\bullet)^\top P_\Dc \mat{c}{\ch \Del_h^\top \h D_{hyw}^\top \\ I} \\
		= (\bullet)^\top q \mat{c}{y_\ast \\ x_\ast \\ r_\ast}^{\hspace{-1.2ex}\top}\hspace{-1ex} \mat{cc}{0 & I \\ -\h C_{hz}^\top & -\h C_{hy}^\top \\ -\h D_{hzr}^\top & -\h D_{hyr}^\top}\! \mat{c}{\ch \Del_h^\top \h D_{hyw}^\top \\ I}
		%M\mat{cc}{I & \h D_{hyw}\ch \Del_h}\mat{ccc}{I& - \h C_{hy} & -\h D_{hyr} \\ 0 & - \h C_{hz} & -\h D_{hzr}}\mat{c}{y_\ast \\ x_\ast \\ r_\ast}q (\bullet)^\top
		\stackrel{\eqref{IOD::eq::data_eq1}}{=} 0.
	\end{multline*}
	By construction and the definition of $\Delf_\Dc$, this identity also holds for $\Del_\ast$ replaced by any $\Del \in \Delf_\Dc$. In other words,
	$\Pb_\Dc$ is a multiplier set for $\{-\h D_{hyw}\ch \Del_h~:~ \Del \in \Delf_\Dc\}$.

	In order to make effective use of $\Pb_\Dc$ for robust analysis, we proceed under the following technical assumption on the describing matrices of the channel from $w$ to $y$ in the LFR \eqref{eq::sys_simple}. We defer a discussion of this hypothesis to the next subsection.
	
	\begin{assu}\label{IOD::ass::kernel_inclusion}
		There exist some $\sigma \in \{1, \dots, h\}$ such that $\ker(\h D_{hyw}) \subset \ker\big(\mat{cc}{\h B_{\sigma w} & 0_{n \times (h-\sigma) n_w}}\big)$ holds.
	\end{assu}
	
	Note that Assumption \ref{IOD::ass::kernel_inclusion} implies $M \h D_{hyw} = (\h B_{\sigma w}, 0)$ and
	\begin{equation}
		M\h D_{hyw} \ch \Del_h
		=  \mat{cc}{\h B_{\sigma w} & 0} \mat{cc}{\ch \Del_{\sigma} & 0\\ \bullet & \bullet}
		= \h B_{\sigma w} \ch \Del_{\sigma}N
		\label{IOD::eq::data_eq2}
	\end{equation}
	by the block-triangular structure of $\ch\Del_h$ and if defining
	\begin{equation*}
		M := \mat{cc}{\h B_{\sigma w} & 0_{n \times (h-\sigma) n_w}}\h D_{hyw}^+
		\teq{ and }
		N:= \mat{cc}{I_{\sigma n_z} & 0}.
	\end{equation*}
	
	We are now in the position to state our first main result.
	\begin{theo}%-------------------------------------------------------------------------
		\label{IOD::theo::noise-free}
		Let $\Pb\subset \S^{n_z + n_w}$ be a multiplier set for $\Delf$, $\Pb_\Dc$ be as in \eqref{IOD::eq::multiplier_noise_free} and let Assumption \ref{IOD::ass::kernel_inclusion} be satisfied. Then the interconnection \eqref{eq::sys_simple} is stable if there exist matrices
		\begin{subequations}
			\label{IOD::theo::eq::LMI}
			\begin{equation}
				X \cg 0,\quad P = \smat{Q & S \\ S^\top & R}\in \Pb \teq{ and } P_\Dc \in \Pb_\Dc
			\end{equation}
			satisfying, with 
			$\h P := \smat{\!\!I_\sigma \kron Q & I_\sigma \kron S \\ I_\sigma \kron S^\top & I_\sigma \kron R}$, the LMI
			\begin{equation}
				(\bullet)^\top \mat{cc|c:c}{X & 0 & & \\ 0 & -X & & \\ \hline & & \h P & \\ \hdashline & & & P_\Dc}
				\mat{cc}{I & 0 \\ -(A^\sigma)^\top & - \h C_{\sigma z}^\top \\ \hline
					0 & I \\ -\h B_{\sigma w}^\top & -\h D_{\sigma zw}^\top \\ \hdashline
					0 & N^\top \\ M^\top & 0} \cg 0.
				\label{IOD::theo::eq::LMIb}
			\end{equation}
		\end{subequations}
	\end{theo}%---------------------------------------------------------------------------
	
	\begin{proof}
		By the structures of $\h P$, $P \in \Pb$ and $\h \Del_{\sigma}$, we have
		\begin{equation}
			(\bullet)^\top \h P \mat{c}{-\h \Del_{\sigma}^\top \\ I}
			= I_\sigma \kron (\bullet)^\top P \mat{c}{-\Del_\ast^\top \\ I} \cle 0.
			\label{IOD::pro::eq::lifted_apriori_multiplier}
		\end{equation}
		This permits us to prove well-posedness and stability as defined in Lemma~\ref{IOD::lem::basic} in the following fashion.
		
		\textit{Well-posedness:} We prove that 
		$(I - \h D_{\sigma zw}^\top\h \Del_{\sigma}^\top)$ is nonsingular, which implies the same 
		for the matrix $I - D_{zw}^\top\Del_\ast^\top$ and, hence, for $I - D_{zw}\Del_\ast$ as well.
		
		Suppose that $(I - \h D_{\sigma zw}^\top\h \Del_{\sigma}^\top)u = 0$ for some $u \neq 0$ and set $v := \h \Del_{\sigma}^\top u$ to infer $u =  \h D_{\sigma zw}^\top v$ and $v\neq 0$. Then we conclude directly from \eqref{IOD::pro::eq::lifted_apriori_multiplier} that
		\begin{equation*}
			(\bullet)^{\!\top} \h P\mat{c}{I \\ \!-\h D_{\sigma zw}^\top}\!v
			=  (\bullet)^{\!\top} \h P\mat{c}{v \\ \!-u}
			=  (\bullet)^{\!\top} \h P\mat{c}{\!-\h \Del_{\sigma}^\top  \\ I}\!u \leq 0.
		\end{equation*}
		Next, let $z_\ast$ be any vertically stacked interconnection signal corresponding to the observed data $\Dc$. Then we infer from the definition of $P_\Dc$ and from the particular block triangular structure of $\h \Del_h$ and $\h D_{h zw}$ the identity
		\begin{equation*}
			\begin{aligned}
				(\bullet)^\top P_\Dc \mat{c}{N^\top \\ 0}v
				&= (\bullet)^\top q\,  \big[\h C_{hz}x_\ast + \h D_{hzr}r_\ast \big]^\top N^\top v \\
				&\stackrel{\mathclap{\eqref{IOD::eq::sys_simple_lifted}}}{=} (\bullet)^\top q\, \big[(I - \h D_{hzw}\h \Del_h)z_\ast \big]^\top N^\top v \\
				&= (\bullet)^\top q\, z_\ast^\top \!\mat{cc}{I - \h \Del_{\sigma}^\top \h D_{\sigma zw}^\top & \bullet \\ 0 & \bullet}\! \mat{c}{I_{\sigma n_z} \\ 0}\! v \\
				&= (\bullet)^\top q\, z_\ast^\top  \mat{c}{I - \h \Del_{\sigma}^\top \h D_{\sigma zw}^\top \\ 0}v = 0.
			\end{aligned}
		\end{equation*}
		Finally, the inequality for the right lower block of \eqref{IOD::theo::eq::LMIb}, $X \cg 0$, $v \neq 0$ and the latter two relations imply
		\begin{equation*}
			0 < (\bullet)^\top \h P \mat{c}{I \\ -\h D_{\sigma zw}^\top}v + (\bullet)^\top P_\Dc \mat{c}{N^\top \\ 0}v
			\leq 0,
		\end{equation*}
		which is a contradiction. Hence $I - \h D_{\sigma zw}^\top \h \Del_{\sigma}^\top$ is nonsingular. 
		
		\textit{Stability:} The key step is to right-multiply
		the LMI \eqref{IOD::theo::eq::LMIb} with $\smat{I \\ \ch \Del_{\sigma}^\top \h B_{\sigma w}^\top}$ and left-multiply its transpose. With the abbreviation $H :=-(I - \h D_{\sigma zw}^\top\h \Del_\sigma^\top)^{-1}\h B_{\sigma w}^\top$, we  note
		\begin{equation*}
			\mat{cc}{I & 0 \\ -(A^\sigma)^\top & - \h C_{\sigma z}^\top}\!
			\mat{c}{I \\ \ch \Del_\sigma^\top \h B_{\sigma w}^\top}
			\!=\!\mat{c}{I \\ -\big(A^\sigma \!+\! \h B_{\sigma w} \ch \Del_\sigma  \h C_{\sigma z}\big)^{\!\top}\!}\!,
		\end{equation*}
		\begin{equation*}
			\mat{cc}{0 & I \\ -\h B_{\sigma w}^\top & -\h D_{\sigma zw}^\top}
			\mat{c}{I \\ \ch \Del_\sigma^\top \h B_{\sigma w}^\top}
			=\mat{c}{- \h \Del_\sigma^\top \\ I} H,
		\end{equation*}
		and 
		\begin{equation}\label{key}
			\mat{cc}{0 & N^{\!\top} \\ M^{\!\top} & 0}\hspace{-0.75ex}
			\mat{c}{I \\ \!\ch \Del_\sigma^\top \h B_{\sigma w}^\top}
			\hspace{-0.6ex}=\hspace{-0.6ex} \mat{c}{N^{\!\top}\ch \Del_\sigma^\top \h B_{\sigma w}^\top \\ M^\top}
			\hspace{-0.6ex}\stackrel{\eqref{IOD::eq::data_eq2}}{=}\hspace{-0.6ex} \mat{c}{\ch \Del_h^\top \h D_{hyw}^\top \\ I}\hspace{-0.6ex}M^{\!\top}\!.
		\end{equation}
		%by \red{exploiting} Assumption \ref{IOD::ass::kernel_inclusion}.
		Consequently, the inequality \eqref{IOD::theo::eq::LMIb}, the full column rank of $\smat{I \\ \ch \Del_{\sigma}^\top \h B_{\sigma w}^\top}$, \eqref{IOD::pro::eq::lifted_apriori_multiplier} and $P_\Dc \in \Pb_\Dc$ lead to
		\begin{multline}\label{key2}
			(\bullet)^{\!\top} \!\mat{cc}{X & 0 \\ 0& -X} \!\mat{c}{I \\ -\big(A^\sigma \!+\! \h B_{\sigma w} \ch \Del_\sigma \h C_{\sigma z}\big)^{\!\top}} \cg \\
			- (\bullet)^{\!\top} \h P\! \mat{c}{- \h \Del_\sigma^\top \\ I}\! H
			- (\bullet)^{\!\top}  P_\Dc\!\mat{c}{\ch \Del_h^\top \h D_{h yw}^\top \\ I}\!M^{\!\top} \!\cge 0.
		\end{multline}
		This Lyapunov inequality in conjunction with $X \cg 0$
		implies that  $A^\sigma + \h B_{\sigma w} \ch \Del_\sigma \h C_{\sigma z}$ is Schur stable and, hence, that the lifted interconnection \eqref{IOD::eq::sys_simple_lifted} is stable. Then the original interconnection \eqref{eq::sys_simple} is stable as well.
	\end{proof}
	
	Let us now highlight the relevance of Assumption \ref{IOD::ass::kernel_inclusion}. It permits to conclude \eqref{key},
	which leads, in turn, to the first strict inequality in \eqref{key2}.
	Stability is assured by the second nonstrict inequality, which exploits
	the prior knowledge as encoded in $\h P$ for $P \in \Pb$ in conjunction with the information learned through data as encoded in $ P_\Dc \in \Pb_\Dc$.

	We emphasize that the novel data-integrated test is always 
	better than or at least as good as the the standard one based on 
	the prior knowledge only. In fact, if $X$ and $P$ satisfy the conditions
	in Lemma~\ref{IOD::lem::basic}, one can show that $X$ and the matrix
	$\hat P$ corresponding to $P$ fulfill those in Theorem~\ref{IOD::theo::noise-free} with $P_\Dc = 0$.
	This observation also permits us to speed up computations, by providing a non-trivial initial guess to solvers such as LMILab \citep{GahNem95} if optimizing over the data-integrated LMIs \eqref{IOD::theo::eq::LMI}.

	\begin{rema}
		\label{IOD::rema::structured_multiplier}
		Theorem \ref{IOD::theo::noise-free} is formulated for the special multiplier set $\{\h P\,:\, P \in \Pb \}$ for clarity.
		It stays true for any larger set of multipliers for $\{\h \Del_{\sigma}\,:\, \Del \in \Delf \}$, which allows us to exploit the diagonal structure of these uncertainties in order to reduce conservatism, at the cost of a higher computational burden.		
	\end{rema}

	\begin{rema}
		\label{IOD::rema::Toeplitz}
		For $o \in \{z, y\}$, we note that $\h C_{ho}x_\ast$ in the data-based equation \eqref{IOD::eq::data_eq1} can be expressed as $\h D_{hox}\h x_\ast$ with 
		\begin{equation*}
			\h D_{hox}\!:=\!\mat{cccc}{C_o & 0 & \dots & 0 \\
				C_oA & C_o & \ddots & \vdots \\
				\vdots & \ddots & \ddots & 0 \\
				C_oA^{h-1} & \dots  & C_oA  & C_o}
			\text{ and }
			\h x_\ast \!:=\! \mat{c}{x_\ast \\ 0 \\ \vdots \\ 0}
			\in \R^{hn}.
		\end{equation*}
		It is then a consequence of the block triangular and block Toeplitz structure of the matrices $\h D_{hyx}, \h D_{hyr}$, etc. that the data-based equation \eqref{IOD::eq::data_eq1} remains valid if we replace the stacked signals $\h x_\ast$, $y_\ast$ and $r_\ast$ with their corresponding Toeplitz matrices; for example, for  $y_\ast$ we would pick
		\begin{equation*}
			\mat{cccc}{y_\ast(0) & 0 & \dots & 0 \\ y_\ast(1) & y_\ast(0) & \ddots & \vdots \\
				\vdots & \ddots & \ddots & 0 \\
				y_\ast(h-1) & \dots & y_\ast(1) & y_\ast(0)} \in \R^{hn_y \times hn_y}.
		\end{equation*}
		We can then replace the scalar variable $q$ in  \eqref{IOD::eq::multiplier_noise_free} with a free matrix variable $Q \in \S^{h}$. This extra freedom  constitutes yet another possibility to improve the stability test in Theorem \ref{IOD::theo::noise-free} at the expense of a higher computational cost.	
	\end{rema}
	
	\begin{rema}
		\label{IOD::rema::multiple}
		Similarly, we can exploit the availability of multiple trajectories $(x_{\ast,j}, r_{\ast,j}, y_{\ast,j})\in \R^n \times \R^{hn_r}\times \R^{hn_y}$ for $j = 1, \dots, \nu$ of the same length by replacing the free scalar variable $q$ in \eqref{IOD::eq::multiplier_noise_free} with a free matrix variable $Q \in \S^\nu$ and the data-vector
		\begin{equation}
			\mat{c}{y_\ast \\ x_\ast \\ r_\ast}  \text{ with } \mat{ccc}{y_{\ast,1} & \dots & y_{\ast, \nu} \\ x_{\ast, 1} & \dots & x_{\ast, \nu} \\ r_{\ast, 1} & \dots & r_{\ast, \nu}} \in \R^{n+h(n_r+n_y) \times \nu}.
			\label{IOD::eq::multiple_data_rema}
		\end{equation}
		This also permits us to recover the findings of \cite{BerSch20} for an LFR with $D_{zw}=0$ and input-state data 
		$\big((r_\ast(k))_{k=0}^{h-1}, (x_\ast(k))_{k=0}^{h}\big)$. 
		Indeed, with $(C_y, D_{yw}, D_{yr}) := (A, B_w, B_r)$ and by splitting the trajectory into multiple ones of length one as
		\begin{equation*}
			(x_\ast(0), r_\ast(0), x_\ast(1)), \dots, (x_\ast(h-1), r_\ast(h-1), x_\ast(h)),
		\end{equation*}
		the matrix in \eqref{IOD::eq::multiple_data_rema} equals $(X_+^\top, X^\top, R^\top)^\top$ and the data equation \eqref{IOD::eq::data_eq1} simplifies to \eqref{IOD::data_based_equation_state_feedback} since no lifting is required. Note that Assumption \ref{IOD::ass::kernel_inclusion} is trivially satisfied with $\sigma = 1$.
	\end{rema}

	\subsection{On Assumption \ref{IOD::ass::kernel_inclusion}}
	
	Intuitively, based on the measurements $y$,	the solution of our problem requires to reconstruct the relevant parts of the uncertain signal $w$ entering the state $x$, at least after some period of time $h - \sigma$. This is indeed reflected by the key relation \eqref{IOD::eq::data_eq2}, which is a consequence of Assumption~\ref{IOD::ass::kernel_inclusion}.
	
	Note that  Assumption \ref{IOD::ass::kernel_inclusion} obeys the following monotonicity property. In particular, if it is satisfied for some horizon length $h$, then it is also satisfied for horizons $\t h > h$.
	
	\begin{lemm}
		Let $\ker(\h D_{hyw}) \subset \ker\big(\mat{cc}{\h B_{\sigma w} & 0}
		\big)$ for some $\sigma \in \{1, \dots, h\}$. Then $\ker(\h D_{h+1,yw}) \subset \ker\big(\mat{cc}{\h B_{\sigma+1, w} & 0}
		\big).$
	\end{lemm}
	
	\begin{proof}
		By the structure of the lifted matrices, we have
		\begin{equation*}
			\mat{c}{\h B_{kw} \\ \hline \h D_{kyw}}
			= \mat{cc}{A^m\h B_{lw} & \h B_{mw} \\ \hline
				\h D_{lyw} & 0 \\ \h C_{my} \h B_{lw} & \h D_{myw}}
		\end{equation*}
		for any $k,l,m \in \N$ with $k = l+m$. Next, let $m:= h-\sigma+1$ and $x = (x_i)_{i=1}^4$ with $x_1 \in \R^{n_w \sigma}$, $x_2,x_4 \in \R^{n_w}$ and $x_3 \in \R^{n_wm}$ satisfy
		$\h D_{h+1,yw}x = 0$. Then we infer
		\begin{equation*}
			0 %= \h D_{h+1,yw} \mat{c}{x \\y \\ z}
			= \mat{c:c}{\h D_{hyw} & 0 \\ C_y \h B_{hw} & D_{yw}}\mat{c}{x_1 \\x_2 \\ x_3\\ \hdashline x_4}
			= \mat{c:c}{\h D_{\sigma yw} & 0 \\ \h C_{m y} \h B_{\sigma w} & \h D_{m yw} } \mat{c}{x_1 \\ \hdashline x_2 \\ x_3 \\ x_4}\!.
		\end{equation*}
		By assumption, this yields $\h B_{\sigma w}x_1 = 0$ and, hence,
		\begin{equation*}
			0 = \mat{c}{0 \\ \h D_{myw}} \mat{c}{x_2 \\ x_3 \\ x_4}
			= \h D_{hyw} \mat{c}{0_{n_w(\sigma - 1)} \\ x_2 \\ x_3 \\ x_4}.
		\end{equation*}
		This time, we have $0 = \h B_{\sigma w}  \smat{0_{n_w(\sigma - 1)} \\ x_2} = B_wx_2$ by the assumption and, in summary, we get
		\begin{equation*}
			\mat{cc}{\h B_{\sigma+1,w} & 0}x
			= \mat{cccc}{A \h B_{\sigma w} & B_w} \mat{c}{x_1 \\ x_2} = 0
		\end{equation*}
		which was to be shown.
	\end{proof}
	
	Note that Assumption \ref{IOD::ass::kernel_inclusion} depends on the particular LFR as used in \eqref{IOD::eq::sys_simple_lifted}. It can often be rendered satisfied by suitable rather immediate adaptations.
	For example, suppose we only know $\ker(\h D_{hyw}) \subset \ker(\h B_{hw_2})$ in some partition $\h B_{hw} = \smat{\h B_{hw_1} \\ \h B_{hw_2}}$, i.e., $M \h D_{hyw} = \h B_{hw_2}$ for some matrix $M$.
	Then we can rewrite the  lifted interconnection \eqref{IOD::eq::sys_simple_lifted} without the channel $r \to y$ as
	\begin{equation*}
		\mat{c}{\t x(k\!+\!1) \\   \h z(k) \\ \h z(k)}
		= \mat{ccc}{A^h & \smat{\h B_{hw_1}\\ 0} & \smat{0 \\ \h B_{hw_2}}  \\
			\h C_{hz} & \h D_{hzw}  &  0  \\
			\h C_{hz} & 0 & \h D_{hyw}}\hspace{-1ex}
		\mat{c}{\t x(k) \\ \h w_1(k) \\ \h w_2(k)}\hspace{-0.7ex},~
		%
		%\h w(k) \!=\! \h \Del_h \h z(k)	
	\end{equation*}
	\begin{equation*}
		\h w_1(k) = \h w_2(k) = \h w(k) = \h \Del_h \h z(k).
	\end{equation*}
	This leads to a corresponding variant of Theorem \ref{IOD::theo::noise-free}
	in which the matrix $\smat{0 & N^\top \\ M^\top & 0}$ is replaced with 
	$\smat{0 & 0 & 0 & I \\ 0 & M^\top & 0 & 0}$. 
	%The proof of the corresponding variant of Theorem \ref{IOD::theo::noise-free} relies on the  multiplication with \red{the matrix}
	The proof 
	%of the corresponding variant of Theorem \ref{IOD::theo::noise-free} 
	then relies on the  multiplication with the matrix
	\begin{equation*}
		\mat{cc}{I & 0 \\ 0 & I \\ \ch \Del_h^\top \h B_{hw_1}^\top & 0 \\ 0 & \ch \Del_h^\top \h B_{hw_2}^\top} \text{ instead of} \mat{c}{I \\ \ch \Del_h^\top \h B_{hw}^\top}.
	\end{equation*}
	%If we multiply this matrix from the left with $\smat{0 & 0 & 0 & I \\ 0 & M^\top & 0 & 0}$ instead of %$\smat{I \\ \ch \Del_h^\top \h B_{hw}^\top}$ with $\smat{0 & N^\top \\ M^\top & 0}$ in %\eqref{IOD::theo::eq::LMIb}, 
	For the product corresponding to \eqref{key}, we then obtain
	\begin{equation*}
		\mat{cc}{0 & \ch \Del_h^\top \h D_{hyw}^\top M^\top \\ 0 & M^\top}
		= \mat{c}{\ch \Del_h^\top \h D_{hyw}^\top \\ I} \mat{cc}{0 & M^\top},
	\end{equation*}
	which permits us to conclude the proof as before.
	
	Finally, notice that a similar assumption is as well required in the results of \cite{BerSch20} if working with the general version of the LFR \eqref{eq::sys_simple}.

	%-----------------------------------------------------------------------------------------
	\subsection{Unknown Initial Condition}\label{IOD::sec::unknown_initial_condition}
	%-----------------------------------------------------------------------------------------

	If the initial condition $x_\ast$ corresponding to the observed finite-horizon input-output trajectory is unknown, we can no longer use the multiplier set $\Pb_\Dc$ in \eqref{IOD::eq::multiplier_noise_free} since it explicitly depends on $x_\ast$.
	In order to construct a useful stability test, it is required to have some prior information about the location of $x_*$ available, which is once again encoded in a suitable family of multipliers for the uncertainty set $\{x_*\}$. % as follows.}
	
	\begin{assu}
		\label{IOD::ass::initial_state}
		$\Pb_x \subset \S^{1+n}$ is a multiplier set for $\{x_\ast\}$, i.e., $\Pb_x$ has an LMI representation and $x_*$ satisfies
		\begin{equation}
			\mat{c}{-x_\ast^\top \\ I}^\top P_x \mat{c}{-x_\ast^\top \\ I} \cle 0 \teq{ for all }P_x \in \Pb_x.
			\label{IOD::ass:eq::initial_state_ineq}
		\end{equation}
	\end{assu}
	
	As a typical example, $x_*$ might be known to be located in some ellipsoid. In other words, $x_*$ satisfies
	%Often, we can at least determine some ellipsoid containing the initial state, i.e., we have some
	\begin{equation*}
		0 \leq \kappa - x_\ast^\top Y x_\ast
		= \mat{c}{1 \\ x_\ast}^\top \mat{cc}{\kappa & 0 \\ 0 & -Y} \mat{c}{1 \\ x_\ast}
	\end{equation*}
	for some fixed $Y \cg 0$ and $\kappa > 0$.
	By the dualization lemma as stated, e.g., in \citep{Sch00}, we conclude that  \eqref{IOD::ass:eq::initial_state_ineq} holds for $\Pb_x := \left\{ q\, \smat{\kappa & 0 \\ 0 & -Y}^{-1}\,:\, q\geq 0\right\}$.

	In order to construct a suitable multiplier set involving all of the available data, we express the data equation \eqref{IOD::eq::data_eq1} as 
	\begin{equation}
		0 =\! \mat{cc}{\h D_{hyw}\ch \Del_h & I}\mat{ccc}{0 & - \h C_{hz} & -\h D_{hzr} \\ I& - \h C_{hy} & -\h D_{hyr}}\mat{cc}{y_\ast & 0 \\ 0 & I_n \\ r_\ast & 0} \!\mat{c}{I \\ x_\ast}.
		\label{IOD::eq::data_eq1b}
	\end{equation}
	Together with Assumption\,\ref{IOD::ass::initial_state}, this yields the multiplier set
	\begin{equation}
		\Pb_\Dc := \left\{(\bullet)^{\!\top} P_x\mat{cc}{y_\ast & 0\\ 0 & I_n \\ r_\ast & 0}^{\hspace{-1.2ex}\top}\hspace{-1ex} \mat{cc}{0 & I \\ -\h C_{hz}^\top & -\h C_{hy}^\top \\ -\h D_{hzr}^\top & -\h D_{hyr}^\top} \,\middle|\,   P_x \in \Pb_x\right\}\!.
		\label{IOD::eq::multiplier_set_noise_free_unknown_x0}
	\end{equation}
	Indeed, by \eqref{IOD::eq::data_eq1b} there exists some matrix $V$ with
	\begin{equation*}
		\mat{cc}{y_\ast & 0\\ 0 & I \\ r_\ast & 0}^{\hspace{-1.2ex}\top}\hspace{-1ex} \mat{cc}{0 & I \\ -\h C_{hz}^\top & -\h C_{hy}^\top \\ -\h D_{hzr}^\top & -\h D_{hyr}^\top}\! \mat{c}{\ch \Del_h^\top \h D_{hyw}^\top \\ I}
		= \mat{c}{-x_\ast^\top \\ I} V
	\end{equation*}
	since the matrix on the left hand side is contained in $\ker\big( \mat{cc}{I & x_\ast^\top}\big)$ and since $\smat{-x_\ast^\top \\ I}$ is a basis matrix of this nullspace. 
	For any multiplier $P_\Dc \in \Pb_\Dc$, we then immediately obtain 
	\begin{equation*}
		(\bullet)^\top P_\Dc \mat{c}{\ch \Del_h^\top \h D_{hyw}^\top \\ I}
		= (\bullet)^\top P_x \mat{c}{-x_\ast^\top \\ I}V
		\stackrel{P_x \in \Pb_x}{\cle} 0.
	\end{equation*}

	This brings us to the following corollary, whose proof proceeds along the lines of the one of Theorem \ref{IOD::theo::noise-free}.
	
	\begin{coro}
		Under Assumption \ref{IOD::ass::initial_state}
		%In the case that the initial condition $x_\ast$ corresponding to the observed data is not known exactly
		and if the multiplier set $\Pb_\Dc$ in \eqref{IOD::eq::multiplier_noise_free} is replaced by \eqref{IOD::eq::multiplier_set_noise_free_unknown_x0},
		Theorem \ref{IOD::theo::noise-free} remains valid.
	\end{coro}
	
	In case that $x_\ast$ is known, we can recover the multiplier set \eqref{IOD::eq::multiplier_noise_free} from \eqref{IOD::eq::multiplier_set_noise_free_unknown_x0} by choosing $\Pb_x := \left\{q \smat{1 & x_\ast^\top \\ x_\ast & x_\ast x_\ast^\top}:\, q\in \R\right\}$. An analysis of the asymptotic behaviour for $\kappa\to\infty$ is left for future research.
	
	%-----------------------------------------------------------------------------------------
	\section{NOISY DATA}\label{IOD::sec::noisy_data}
	%-----------------------------------------------------------------------------------------
	
	Let us now consider the following noisy version of the  feedback interconnection \eqref{eq::sys_simple}, which involves an additional noise input $n$ and an error signal $e$:
	\begin{equation}
		\begin{aligned}
			\label{IOD::eq::sys}
			\mat{c}{x(k+1) \\   z(k) \\ e(k) \\  y(k)}
			&= \mat{cccc}{A & B_w & B_n & B_r  \\
				C_z & D_{zw} & D_{zn} & D_{zr}  \\
				C_e & D_{ew} & D_{en} & D_{er} \\
				C_y & D_{yw} & D_{yn} & D_{yr}}\!
			\mat{c}{x(k) \\  w(k) \\ n(k) \\ r(k)},\\
			w(k) &= \Del_\ast z(k).
		\end{aligned}
	\end{equation}
	Without any knowledge about the noise, the goal is to derive a data-enhanced stability test for \eqref{IOD::eq::sys} and, concurrently, determine a good upper bound on the energy gain
	\begin{equation}\label{gain}
		\sup_{d\in \ell_2^{n_d} \setminus \{0\}, ~ x(0) = 0} \frac{\|e\|_{\ell_2}}{\|d\|_{\ell_2}}
	\end{equation}
	of its performance channel, where $d := \smat{n \\ r}$ denotes the so-called generalized disturbance. In the sequel, we use the abbreviations $B_d := (B_n\ B_r)$, 
	$D_{zd} := (D_{zn}\ D_{zr})$ and $D_{ed} := (D_{en}\ D_{er})$.		
	
	%-------------------------------------------------------------------------------------
	\subsection{A Basic Robust Performance Analysis Result}
	%-------------------------------------------------------------------------------------
	
	Without taking any observed data into account, stability and performance of the uncertain interconnection \eqref{IOD::eq::sys} can again be analyzed with the full-block S-procedure. 
	In view of \eqref{gain}, we pick the performance index $P_\ga := \smat{I & 0 \\ 0 & -\ga^{2} I}$ and recall the following result from \citep{Sch00}.

	\begin{lemm}
		\label{IOD::lem::basic_perf}
		Let $\Pb$ be a multiplier set for $\Delf$. Then the interconnection \eqref{IOD::eq::sys} is stable and the energy gain \eqref{gain} is bounded by $\ga>0$ if there exist matrices %$X\in \S^n$ and $P\in \S^{n_z + n_w}$ satisfying
		\begin{subequations}
			\label{IOD::lem::eq::basic_perf}
			\begin{equation}
				X \cg 0 \teq{ and } P \in \Pb
			\end{equation}
			satisfying
			\begin{equation}
				\label{IOD::lem::eq::basic_perfb}
				(\bullet)^\top \mat{cc|c:c}{X & 0 & & \\ 0 & -X & &\\ \hline & & P& \\ \hdashline & & & P_\ga^{-1}}\mat{ccc}{I & 0 & 0\\ -A^\top & -C_z^\top & -C_e^\top\\ \hline 0 & I & 0\\ -B_w^\top & -D_{zw}^\top & -D_{ew}^\top \\ \hdashline 0 & 0 & I \\ -B_d^\top & -D_{zd}^\top & -D_{ed}^\top} \cg 0.
			\end{equation}
		\end{subequations}
	\end{lemm}
	
	After using the Schur complement, one can minimize $\ga^2$ under the LMI constraints \eqref{IOD::lem::eq::basic_perf} in order to determine the best upper bound on the energy gain which can be guaranteed by multipliers in the set $\Pb$.
	
	%-------------------------------------------------------------------------------------
	\subsection{Incorporating Data}
	%-------------------------------------------------------------------------------------
	
	Once again, we target at enhancing the robust stability and performance test of Lemma~\ref{IOD::lem::basic_perf} by incorporating knowledge of some finite-time horizon trajectory \eqref{data} generated by
	the interconnection \eqref{IOD::eq::sys}. 
	%\begin{equation*}
	%		\Dc := \big(x_\ast, (n_\ast(k))_{k=0}^{h-1},  (r_\ast(k))_{k=0}^{h-1}, (y_\ast(k))_{k=0}^{h-1}\big),
	%	\end{equation*}
	%\red{We now} assume that a noisy input-output trajectory on a finite horizon is available which results from the interconnection \eqref{eq::sys_simple} initialized in a known point $x_\ast$. The involved data is
	The noise sequence $n_\ast$ is not available, but we assume that $x_\ast$ is known to simplify the exposition. 
	
	As before, we formally introduce the set
	\begin{equation*}
		\Delf_\Dc := \left\{\Del~:~ \begin{array}{l}\text{The measured output of \eqref{IOD::eq::sys} with} \\ \text{$x(0) = x_\ast$ and $\Del_\ast$ replaced by $\Del$} \\ \text{in response to the input $\smat{n_\ast \\ r_\ast}$ is $y_\ast$}
		\end{array} \right\}
	\end{equation*}
	of all constant parametric uncertainties that are compatible with the data $\Dc$.
	
	Similarly as \cite{BerSch20}, we assume that some information on the noise sequence $n_\ast$ is available. In contrast to other approaches in the literature, however, we work with constraints imposed on the vertically stacked noise sequence.		
	\begin{assu}
		\label{IOD::ass::noise}
		We have a multiplier set $\Pb_n$ for $\{n_\ast\}$.
	\end{assu}		
	Dedicated examples of sets $\Pb_n$ encoding properties of the noise $n_\ast$ can be extracted from \citep{BerSch20}.
	
	%-------------------------------------------------------------------------------------
	\subsection{Main Result}
	%-------------------------------------------------------------------------------------
	
	As in the noise-free case, we get the following data-based equation by lifting the signals in the interconnection \eqref{IOD::eq::sys}:
	\begin{equation}
		\label{IOD::eq::data_eq3}
		\arraycolsep=2pt
		\begin{aligned}
			0 =&\, y_\ast  \!-\!\big[\h C_{hy}\! +\! \h D_{hyw}\ch \Del_h\h C_{hz} \big]x_\ast \!-\!\big[\h D_{hyn}\! +\! \h D_{hyw}\ch \Del_h \h D_{hzn} \big]n_\ast \\
			&\phantom{y_\ast}-\!\big[\h D_{hyr} \!+\! \h D_{hyw}\ch \Del_h\h D_{hzr} \big] r_\ast\\
			=&\, \mat{cc}{\h D_{hyw}\ch \Del_h & I}
			\mat{cccc}{0 & - \h C_{hz} & -\h D_{hzn} & -\h D_{hzr}\\ I& - \h C_{hy} & -\h D_{hyn}& -\h D_{hyr}}\mat{c}{y_\ast \\ x_\ast \\ n_\ast \\ r_\ast}\\
			=&\,\mat{cc}{\h D_{hyw}\ch \Del_h & I}\hspace{-0.5ex}
			\mat{cccc}{0 & - \h C_{hz} & -\h D_{hzn} & -\h D_{hzr}\\ I& - \h C_{hy} & -\h D_{hyn}& -\h D_{hyr}}\hspace{-1ex}\mat{cc}{y_\ast & 0 \\ x_\ast & 0 \\ 0 & I \\ r_\ast & 0} \hspace{-1ex}\mat{c}{I \\ n_\ast}\!.
		\end{aligned}
	\end{equation}
	Based on this equation and Assumption \ref{IOD::ass::noise}, we construct the following multiplier set involving all of the available data as in Subsection \ref{IOD::sec::unknown_initial_condition}:
	\begin{equation}
		\Pb_\Dc\! :=\! \left\{(\bullet)^{\!\top} P_n\mat{cc}{y_\ast & 0 \\ x_\ast &  0 \\ 0 & I \\ r_\ast & 0}^{\hspace{-1.2ex}\top}\hspace{-1ex} \mat{cc}{0 & I \\ -\h C_{hz}^\top & -\h C_{hy}^\top \\ -\h D_{hzn}^\top & -\h D_{hyn}^\top \\ -\h D_{hzr}^\top & -\h D_{hyr}^\top} \,\middle|\,   P_n \in \Pb_n \right\}\!.
		\label{IOD::eq::multiplier_set_noisy}
	\end{equation}
	%	Indeed, by construction any matrix $P_\Dc \in \Pb_\Dc$ satisfies
	%	\begin{equation*}
	%		\begin{aligned}
	%			&\phantom{=}~(\bullet)^\top P_\Dc \mat{c}{\ch \Del_h^\top \h D_{hyw}^\top \\ I} \\
	%			%
	%			&= (\bullet)^\top P_n \mat{cc}{y_\ast & 0\\ x_\ast & 0 \\ 0 & I \\ r_\ast & 0}^{\hspace{-1.2ex}\top}\hspace{-1ex} \mat{cc}{0 & I \\ -\h C_{hz}^\top & -\h C_{hy}^\top \\ -\h D_{hzn}^\top & -\h D_{hyn}^\top \\ -\h D_{hzr}^\top & -\h D_{hyr}^\top}\! \mat{c}{\ch \Del_h^\top \h D_{hyw}^\top \\ I} \\
	%			%
	%			&= (\bullet)^\top P_n \mat{cc}{y_\ast & 0\\ x_\ast & 0 \\ 0 & I \\ r_\ast & 0}^{\hspace{-1.2ex}\top}\hspace{-1ex} \mat{c}{I \\ -\big[\h C_{hy} + \h D_{hyw}\ch \Del_h\h C_{hz}\big]^\top  \\ -\big[\h D_{hyn} + \h D_{hyw}\ch \Del_h \h D_{hzn}\big]^\top \\ -\big[\h D_{hyr} + \h D_{hyw}\ch \Del_h \h D_{hzr}\big]^\top} \\
	%			%
	%			&\stackrel{\mathclap{\eqref{IOD::eq::data_eq3}}}{=} (\bullet)^\top P_n \mat{c}{-n_\ast^\top \\ I} \big[\h D_{hyn} + \h D_{hyw}\ch \Del_h \h D_{hzn} \big]^\top
	%
	%			\stackrel{P_n \in \Pb_n}{\cle} 0.
	%		\end{aligned}
	%	\end{equation*}
	%	
	%	\gr{A word on the conceptual way how to build the multiplier set?}
	In expanding Assumption \ref{IOD::ass::kernel_inclusion}, we now require to be able to extract from $y$ suitable information about how $w$ affects both the state $x$ and the error $e$ in the 
	interconnection \eqref{IOD::eq::sys}. This leads to the following 
	assumption on the matrices $\h D_{hyw}$ and  $\h B_{hw}$, $\h D_{hew}$ of the system \eqref{IOD::eq::sys} after lifting.
	
	\begin{assu}
		\label{IOD::ass::kernel_inclusion2}
		There exists some $\sigma \in \{1, \dots, h\}$ such that the inclusions
		$\ker(\h D_{hyw}) \subset \ker\left(\mat{cc}{\h B_{\sigma w} & 0_{n \times (h-\sigma)n_w}} \right)$ and
		$\ker(\h D_{hyw}) \subset \ker\left(\mat{cc}{ \h D_{\sigma ew} & 0_{n_e \times (h-\sigma)n_w}} \right)$ hold true.
	\end{assu}
	
	Note again that Assumption \ref{IOD::ass::kernel_inclusion2} and the block triangular structure of $\ch \Del_h$ imply that the identity
	\begin{equation}
		\mat{c}{M_b \\ M_d}\h D_{hyw} \ch \Del_h
		%
		%=  \mat{cc}{\h B_{\sigma w} & 0 \\ \h D_{\sigma ew} & 0} \mat{cc}{\ch \Del_{\sigma} & 0\\ \bullet & \bullet}
		%
		= \mat{c}{\h B_{\sigma w}\\ \h D_{\sigma ew}} \ch \Del_{\sigma}N
		\label{IOD::eq::data_eq4}
	\end{equation}
	is valid for the matrices
	\begin{equation*}
		\mat{c}{M_b \\ \hline M_d} \!:=\! \mat{cc}{\h B_{\sigma w} & 0_{n \times (h-\sigma) n_w} \\ \hline \!\h D_{\sigma ew} & 0_{\sigma n_e \times (h-\sigma)n_w}}\!\h D_{hyw}^+
		\text{ and }
		N:= \mat{cc}{I_{\sigma n_z} & 0}.
	\end{equation*}

	This leads to our second main result.
	
	\begin{theo}%-------------------------------------------------------------------------
		\label{IOD::theo::noisy}
		Let $\Pb$ be a multiplier set for $\Delf$ and let Assumptions \ref{IOD::ass::noise} and \ref{IOD::ass::kernel_inclusion2} be satisfied. Moreover, let $\Pb_\Dc$ be given by \eqref{IOD::eq::multiplier_set_noisy}. Then the interconnection \eqref{eq::sys_simple} is stable and the energy gain \eqref{gain} is bounded by $\ga > 0$ if there exist %symmetric matrices $X$, $P=\smat{Q & S \\ S^\top & R}$ and $P_\Dc$ satisfying
		\begin{subequations}
			\begin{equation}
				X \cg 0,\quad P =\smat{Q & S \\ S^\top & R}\in \Pb \teq{ and } P_\Dc \in \Pb_\Dc
			\end{equation}
			satisfying, with 
			$\h P := \smat{\!\!I_\sigma \kron Q & I_\sigma \kron S \\ I_\sigma \kron S^\top & I_\sigma \kron R}$, the LMI		
			\begin{equation}
				(\bullet)^{\!\top}\hspace{-0.7ex} \mat{cc|c:c:c}{X & 0 & & &\\ 0 & -X & & &\\ \hline & & \h P & &\\ \hdashline & & & P_\Dc & \\ \hdashline & & & & P_\ga}
				\hspace{-1ex}
				\mat{ccc}{I & 0 & 0\\ -(A^\sigma)^\top & - \h C_{\sigma z}^\top & -\h C_{\sigma e}^\top\\ \hline
					0 & I & 0 \\ -\h B_{\sigma w}^\top & -\h D_{\sigma zw}^\top & -\h D_{\sigma ew}^\top \\ \hdashline
					0 & N^\top & 0\\ M_b^\top & 0 & M_d^\top \\ \hdashline 0 & 0 & I \\ -\h B_{\sigma d}^\top & -\h D_{\sigma zd}^\top & -\h D_{\sigma ed}^\top} \!\cg\! 0.
				\label{IOD::theo::eq::LMIb_noisy}
			\end{equation}
		\end{subequations}
	\end{theo}%---------------------------------------------------------------------------

	The proof follows the same lines as the one of Theorem \ref{IOD::theo::noise-free}, but now based on the following extension of  \eqref{key}:
	%relies on multiplying  \eqref{IOD::theo::eq::LMIb_noisy} with the matrix
	%\begin{equation*}
	%		\mat{cc}{I & 0 \\ \ch \Del_\sigma \h B_{\sigma w}^\top & \ch \Del_{\sigma}\h D_{\sigma ew}^\top \\ 0 & I}.
	%	\end{equation*}
	%	Assumption \ref{IOD::ass::kernel_inclusion2} permits us to infer
	\begin{multline*}
		\mat{ccc}{0 & N^\top & 0\\ M_b^\top & 0 & M_d^\top}	\mat{cc}{I & 0 \\ \ch \Del_\sigma \h B_{\sigma w}^\top & \ch \Del_{\sigma}\h D_{\sigma ew}^\top \\ 0 & I}\\
		= \mat{c}{N^\top \ch \Del_{\sigma} \mat{cc}{\h B_{\sigma w}^\top & \h D_{\sigma ew}^\top} \\ \mat{cc}{M_b^\top& M_d^\top}}
		\stackrel{\eqref{IOD::eq::data_eq4}}{=} \mat{c}{\h \Del_h^\top \h D_{hyw}^\top\\ I} \mat{c}{M_b \\ M_d}^{\!\top}.
	\end{multline*}
	This permits to exploit \eqref{IOD::eq::data_eq3} and \eqref{IOD::eq::multiplier_set_noisy} as earlier.

	\begin{rema}
		\label{IOD::rema::toeplitz_multiplier}
		Following Remark \ref{IOD::rema::Toeplitz}, we can reduce the conservatism in Theorem \ref{IOD::theo::noisy} by considering Toeplitz matrices for the available data. If
		$N$ denotes the Toeplitz matrix corresponding to the noise signal $n_\ast$, this requires to work with a multiplier set for $\{N\}$.
		In case of $\|n_\ast(k)\|_2 \leq \eps$ for all $k \in \{0, \dots, h-1\}$, one choice is
		\begin{equation*}
			\left\{\mat{cc}{\diag(\la_1, \dots, \la_h) & 0 \\ 0 & -\sum_{k = 1}^{h} \la_k\eps^2 f_k \smat{0 & 0 \\ 0 & I_{g_k}}} \colon \la_k \geq 0 \right\}
		\end{equation*}
		with $f_k := h\!-\!k\!+\!1$, $g_k := n_n f_k$ and $n_n$ the length of $n_\ast(k)$.
	\end{rema}

	\begin{rema}
		\label{IOD::rema::comp_burden}
		Depending on the strategies in Remarks \ref{IOD::rema::structured_multiplier} and  \ref{IOD::rema::Toeplitz}, the computational burden can increase quite drastically with the horizon length $h$. As a remedy, one can artificially split too long data trajectories into multiple smaller ones and exploit Remark \ref{IOD::rema::multiple}. This is much more affordable, but requires estimates of the start of each state trajectory.
	\end{rema}

	%-----------------------------------------------------------------------------------------
	\section{EXAMPLE}\label{IOD::sec::example}
	%-----------------------------------------------------------------------------------------
	
	\begin{figure}
		\begin{center}
			\includegraphics[width=0.49\textwidth]{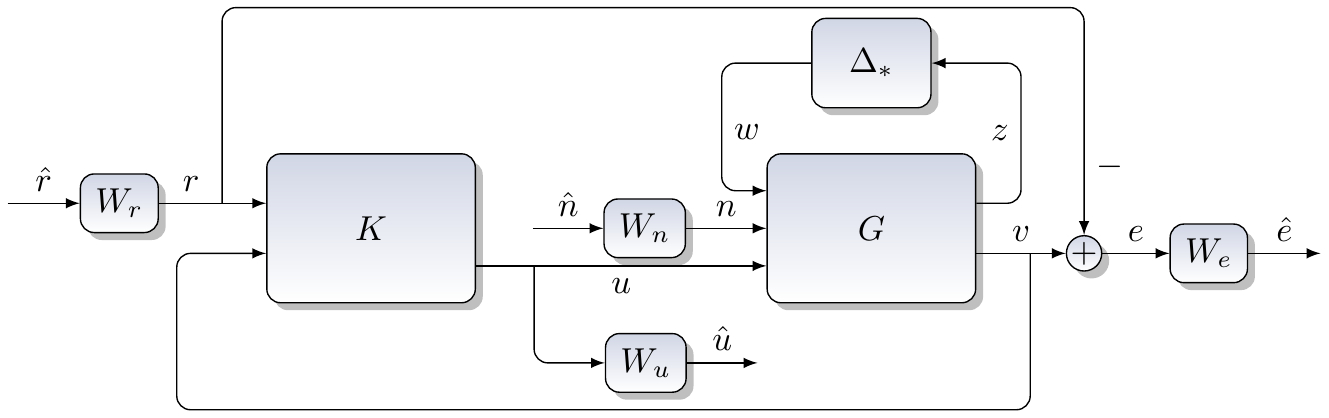}
		\end{center}
		\vspace{-1ex}
		\caption{A standard weighted tracking configuration.}
		\label{NSY::fig::tracking_config}
	\end{figure}

	Let us consider a simple model of a flexible satellite explained in \citep{FraPow10}  and given by
	\begin{equation}
		\mat{c}{\dot {\t x}(t) \\ \hline v(t)}
		= \mat{cccc|c:c}{0 & 1 & 0 & 0 & 0 & 0 \\ -\frac{k}{J_2} & -\frac{b}{J_2} & \frac{k}{J_2} & \frac{b}{J_2} & 1 & 0 \\ 0 & 0 & 0 & 1 & 0 & 0 \\ \frac{k}{J_1} & \frac{b}{J_1} & -\frac{k}{J_1} & -\frac{b}{J_1} & 0 & \frac{1}{J_1} \\ \hline 1 & 0 & 0 & 0 & 0 & 0}
		\mat{c}{\t x(t) \\ \hline n(t) \\ u(t)}.
		\label{DI::eq::sys_ex}
	\end{equation}
	Here, the state is $\t x = \col(\theta_2, \dot \theta_2, \theta_1, \dot \theta_1)$ and the constants are $J_1 = 1$, $J_2 = 0.1$, $k_* = 0.091$ and $b_* = 0.0036$. Now suppose that $k_*$ and $b_*$ are merely known to satisfy
	\begin{equation}
		k_* \in [0.08, 0.12]
		\teq{ and }
		b_* \in [0.0034, 0.02].
		\label{IOD::eq::unc_intervals}
	\end{equation}
	Next, we generate an LFR of this model with the uncertain parameters $k$ and $b$ \citep{ZhoDoy98} and discretize the resulting generalized plant via zero-order hold with a sampling time of $0.05$ seconds.
	Based on the configuration in Fig.~\ref{NSY::fig::tracking_config} with the weights
	\begin{equation*}
		W_r = 1,\quad
		W_n = 0.4, \quad
		W_u = 0.1,\quad
		W_{e}(z)= \frac{z-0.9567}{2z-2}, %\frac{0.5s + 0.433}{s + 0.00433}
	\end{equation*}
	we design a dynamic $H_\infty$ controller $K$ for the parameters
	\begin{equation*}
		k = 0.1\neq k_*
		\teq{ and }
		b = 0.0117\neq b_*.
	\end{equation*}
	The determined controller is still stabilizing the actual plant (with the parameters $k_*$ and $b_*$)	and the energy gain of the actual closed-loop system is bounded by
	\begin{equation*}
		1.1588.
	\end{equation*}
	Without knowledge of $k_*$, $b_*$, this bound can actually not be determined. Instead, we estimate a gain bound based on Theorem~\ref{IOD::theo::noisy}.	To this end, We include the measured signal $y := \smat{r \\ v}$ as an additional output and model our system as \eqref{IOD::eq::sys} with $\Del_\ast = \smat{k & 0 \\ 0 & b} \in \Delf := \{\smat{k & 0 \\ 0 & b}~\colon~ \eqref{IOD::eq::unc_intervals} \}$.
	A classical robustness anaylsis relying on Lemma \ref{IOD::lem::basic_perf} with DG-scalings \citep{Sch00} leads to the much larger upper bound
	\begin{equation*}
		9.8433.
	\end{equation*}
	Note that all computations are performed with LMILab \citep{GahNem95} and all corresponding files can be found under the doi: \href{https://doi.org/10.5281/zenodo.7761837}{10.5281/zenodo.7761837}.
	Next, we assume that we have available an input output trajectory
	of the closed-system being initialized in zero and affected by noise satisfying $\|n_\ast(k)\|_2 \leq \eps$ for all $k$. 
	The employed reference signal is $r(t) = 3\chi_{[0, 1]}(t) - 2\chi_{[1.5, 3]}(t)$, where $\chi_I$ denotes the indicator function on the interval $I$. Note that the choice of $r$ influences the computed upper bounds.
	
	Since Assumption \ref{IOD::ass::kernel_inclusion2} is satisfied for $\sigma = h-1$, we can employ Theorem \ref{IOD::theo::noisy} with DG-scalings for $\Delf$ and the noise multipliers as in Remark \ref{IOD::rema::toeplitz_multiplier}.
	This leads to the following upper bounds for several data-lengths and noise bounds $\epsilon$:
	
	\begin{center}
		\begin{tabular}{cccccc}
			$\eps$ & $h = 10$ & $h = 15$ & $h = 20$ & $h = 30$ & $h = 40$ \\ \hline
			%$0.1$ &    $7.9366$ &   $5.9087$   & $5.0922$  &  $4.7167$  &  $4.4440$ \\
			%$0.05$ & $8.4756$ &   $7.4544$ &  $5.8190$ & $5.1444$  & $4.8231$ \\
			%$0.01$ & $5.8508$ &   $4.8302$ &  $4.2039$ & $3.8206$  & $3.5149$\\
			%$0.1$ & $8.4161$ & $7.2383$ & $6.3681$ & $5.8003$ & $5.3449$ \\
			%$0.05$ & $7.3126$ & $6.0117$ & $5.2240$ & $4.7371$ & $4.3614$ \\
			%$0.01$ & $5.6778$ & $4.5474$ & $3.5819$ & $2.9345$ & $2.5745$ \\
			0.1 & 5.9550  &  4.6522 &   3.8936  &  3.0641  &  2.9674 \\
			0.05 & 5.5826 &   4.4327 &   3.7015 &   2.8915  &  2.8117 \\
			0.01 & 5.2073 &   4.1293  &  3.4027 &   2.6295   & 2.5675\\
		\end{tabular}
	\end{center}
	
	% Times for eps = 0.1: 5.4, 19.5, 32.9, 169.3, 656.4
	
	These results show that we can indeed improve our robust analysis test based on data, even if this data does not involve the full state and the LFR describing the underlying system is general.
	However, as is, this enhancement can be computationally rather demanding (see Remark \ref{IOD::rema::comp_burden}); for example, the computation of the upper bounds for $\eps = 0.1$ took about $0.1, 5.4, 19.5, 32.9, 169.3$ and $656.4$ seconds on a general purpose desktop computer.
	
	%-----------------------------------------------------------------------------------------
	\section{CONCLUSIONS}\label{IOD::sec::conclusions}
	%-----------------------------------------------------------------------------------------
	
	We present an approach to systematically enhance robust analysis tests by incorporating measured input and output data, by benefiting from well-established and powerful tools in robust control. In contrast to most of the existing purely data-driven approaches, we employ linear fractional representations, which offer a high flexibility for including available prior physical knowledge about the underlying system. It is a key novelty that we do not require the system state to be part of the gathered data. 
	Technically, we rely on a lifted version of the system description and show how to learn information about the system's unknown components through constructing multiplier classes based on the gathered data. A simple numerical example illustrates the benefit to substantially reduce conservatism by taking the available data into account. 
	
	Future efforts are devoted to a reduction of the computational burden for large horizons and the inclusion of unknown system components beyond time-invariant parameters.

	%\begin{ack}
	%Place acknowledgments here.
	%\end{ack}
	\renewcommand\refname{REFERENCES}
	\bibliography{literatur}  % bib file to produce the bibliography

\begin{thebibliography}{21}
\providecommand{\natexlab}[1]{#1}
\providecommand{\url}[1]{\texttt{#1}}
\providecommand{\urlprefix}{URL }
\expandafter\ifx\csname urlstyle\endcsname\relax
  \providecommand{\doi}[1]{doi:\discretionary{}{}{}#1}\else
  \providecommand{\doi}{doi:\discretionary{}{}{}\begingroup
  \urlstyle{rm}\Url}\fi

\bibitem[{Berberich et~al.(2022)Berberich, Scherer, and
  Allg{\"o}wer}]{BerSch20}
Berberich, J., Scherer, C.W., and Allg{\"o}wer, F. (2022).
\newblock Combining prior knowledge and data for robust controller design.
\newblock \emph{{IEEE} Trans. Autom. Control},
\newblock 1--16.

\bibitem[{Berkenkamp and Schoellig(2015)}]{BerSch15}
Berkenkamp, F. and Schoellig, A.P. (2015).
\newblock Safe and robust learning control with gaussian processes.
\newblock
\newblock In \emph{Proc. Eur. Control Conf.}

\bibitem[{Chen and Francis(1995)}]{CheFra95}
Chen, T. and Francis, B.A. (1995).
\newblock \emph{Optimal Sampled-Data Control Systems}.
\newblock
\newblock Springer-Verlag London.

\bibitem[{De~Persis and Tesi(2020)}]{DeTes20}
De~Persis, C. and Tesi, P. (2020).
\newblock Formulas for data-driven control: Stabilization, optimality, and
  robustness.
\newblock \emph{{IEEE} Trans. Autom. Control},
\newblock 65(3), 909--924.

\bibitem[{Doyle et~al.(1991)Doyle, Packard, and Zhou}]{DoyPac91}
Doyle, J., Packard, A., and Zhou, K. (1991).
\newblock Review of {LFT}s, {LMI}s, and $\mu$.
\newblock In \emph{Proc. 30th IEEE Conf. Decision and Control},
\newblock 1227--1232.

\bibitem[{Fiedler et~al.(2021)Fiedler, Scherer, and Trimpe}]{FieSch21}
Fiedler, C., Scherer, C.W., and Trimpe, S. (2021).
\newblock Learning-enhanced robust controller synthesis with rigorous
  statistical and control-theoretic guarantees.
\newblock
\newblock In \emph{Proc. 60th {IEEE} Conf. Decision and Control}.

\bibitem[{Franklin et~al.(2010)Franklin, Powell, and Emami-Naeini}]{FraPow10}
Franklin, G.F., Powell, J.D., and Emami-Naeini, A. (2010).
\newblock \emph{Feedback Control of Dynamic Systems}.
\newblock
\newblock Pearson.

\bibitem[{Gahinet et~al.(1995)Gahinet, Nemirovski, Laub, and
  Chilali}]{GahNem95}
Gahinet, P., Nemirovski, A., Laub, A.J., and Chilali, M. (1995).
\newblock {LMI} control toolbox: {F}or use with {M}atlab.
\newblock Technical report,
\newblock The MathWorks, Inc.

\bibitem[{Holicki et~al.(2021)Holicki, Scherer, and Trimpe}]{HolSch21}
Holicki, T., Scherer, C.W., and Trimpe, S. (2021).
\newblock Controller design via experimental exploration with robustness
  guarantees.
\newblock \emph{{IEEE} Control Syst. Lett.},
\newblock 5(2), 641--646.

\bibitem[{Horn and Johnson(1991)}]{HorJoh91}
Horn, R.A. and Johnson, C.R. (1991).
\newblock \emph{Topics in Matrix Analysis}.
\newblock
\newblock Cambridge Univ. Press.

\bibitem[{Hou and Wang(2013)}]{HouWan13}
Hou, Z.S. and Wang, Z. (2013).
\newblock From model-based control to data-driven control: {S}urvey,
  classification and perspective.
\newblock \emph{Inform. Sciences},
\newblock 235, 3--35.

\bibitem[{Kober et~al.(2013)Kober, Bagnell, and Peters}]{KobBag13}
Kober, J., Bagnell, J.A., and Peters, J. (2013).
\newblock Reinforcement learning in robotics: A survey.
\newblock \emph{Int. J. Robot. Res.},
\newblock 32(11), 1238--1274.

\bibitem[{Marco et~al.(2017)Marco, Hennig, Schaal, and Trimpe}]{MarHen17}
Marco, A., Hennig, P., Schaal, S., and Trimpe, S. (2017).
\newblock On the design of {LQR} kernels for efficient controller learning.
\newblock
\newblock In \emph{Proc. 56th IEEE Conf. Decision and Control}.

\bibitem[{Markovsky and D{\"o}rfler(2021)}]{MarDoe21}
Markovsky, I. and D{\"o}rfler, F. (2021).
\newblock Behavioral systems theory in data-driven analysis, signal processing,
  and control.
\newblock \emph{Annual Reviews in Control},
\newblock 52, 42--64.

\bibitem[{Megretsky and Rantzer(1997)}]{MegRan97}
Megretsky, A. and Rantzer, A. (1997).
\newblock System analysis via integral quadratic constraints.
\newblock \emph{IEEE Trans. Autom. Control},
\newblock 42(6), 819--830.

\bibitem[{Scherer(2000)}]{Sch00}
Scherer, C.W. (2000).
\newblock Robust mixed control and linear parameter-varying control with full
  block scalings.
\newblock In L.~El~Ghauoui and S.I. Niculescu (eds.), \emph{Advances in linear
  matrix inequality methods in control}.
\newblock SIAM.

\bibitem[{Scherer(2005)}]{Sch05}
Scherer, C.W. (2005).
\newblock Relaxations for robust linear matrix inequality problems with
  verifications for exactness.
\newblock \emph{{SIAM} J. Matrix Anal. Appl.},
\newblock 27(2), 365--395.

\bibitem[{van Waarde et~al.(2022{\natexlab{a}})van Waarde, Camlibel, and
  Mesbahi}]{WaaCam22}
van Waarde, H.J., Camlibel, M.K., and Mesbahi, M. (2022{\natexlab{a}}).
\newblock From noisy data to feedback controllers: Nonconservative design via a
  matrix s-lemma.
\newblock \emph{{IEEE} Trans. Autom. Control},
\newblock 67(1), 162--175.

\bibitem[{van Waarde et~al.(2022{\natexlab{b}})van Waarde, Eising, Camlibel,
  and Trentelman}]{WaaEis22}
van Waarde, H.J., Eising, J., Camlibel, M.K., and Trentelman, H.L.
  (2022{\natexlab{b}}).
\newblock
\newblock
\newblock A behavioral approach to data-driven control with noisy input-output
  data.

\bibitem[{von Rohr et~al.(2018)von Rohr, Trimpe, Marco, Fischer, and
  Palagi}]{RohTri18}
von Rohr, A., Trimpe, S., Marco, A., Fischer, P., and Palagi, S. (2018).
\newblock Gait learning for soft microrobots controlled by light fields.
\newblock
\newblock In \emph{Proc. Int. Conf. Intell. Robots Syst.}

\bibitem[{Zhou and Doyle(1998)}]{ZhoDoy98}
Zhou, K. and Doyle, J.C. (1998).
\newblock \emph{Essentials of Robust Control}, volume 104.
\newblock
\newblock Prentice Hall.

\end{thebibliography}
	% with bibtex (preferred)
	
	%		\appendix
	
	%	\section{Auxiliary Results and Technical Proofs}    % Each appendix must have a short title.
	
\end{document}